\documentclass{amsart}

\usepackage{amsmath, amssymb, tikz, color}
\usetikzlibrary{calc,through,intersections, arrows, decorations.markings, matrix}
\tikzset{>=latex}

\newcommand{\spaces}{\mathbf{Spaces}}

\newcommand{\integers}{\mathbb{Z}}
\newcommand{\map}{\mathrm{Map}}
\newcommand{\ext}{\mathrm{Ext}}
\renewcommand{\hom}{\mathrm{Hom}}
\newcommand{\isom}{\cong}
\newcommand{\sing}{\mathrm{Sing}}

\newcommand{\DD}{\mathbf{D}}
\newcommand{\CC}{\mathbf{C}}

\newcommand{\TT}{\mathbf{T}}
\newcommand{\alg}{\mathbf{Alg}}
\newcommand{\halg}{\mathbf{hAlg}}
\newcommand{\sSet}{\mathbf{sSet}}
\renewcommand{\top}{\mathbf{Top}}
\newcommand{\ho}{\mathrm{Ho}}

\newtheorem{theorem}{Theorem}[section]
\newtheorem{lemma}[theorem]{Lemma}
\newtheorem{proposition}[theorem]{Proposition}
\newtheorem{corollary}[theorem]{Corollary}
\newtheorem{example}[theorem]{Example}

\theoremstyle{definition}

\newtheorem{definition}[theorem]{Definition}

\begin{document}

\title{Detecting Mapping Spaces}
\author{Alyson Bittner}
\begin{abstract}
We show if $A$ is a finite CW-complex such that algebraic theories detect mapping spaces out of $A$, then $A$ has the homology type of a wedge of spheres of the same dimension. Furthermore, if $A$ is simply connected then $A$ has the homotopy type of a wedge of spheres.
\end{abstract}
\maketitle

\section{Introduction}
\label{INTRO SEC}

 The problem of studying the structure of mapping spaces has  rich history in 
 homotopy theory. The main goal here is to find out, for a given a space $A$, how 
 one can characterize spaces of the homotopy type of pointed 
 mapping spaces $\map_{*}(A, Y)$.
 An example of a powerful result which solves this problem in a special case
 is the Sullivan conjecture proved by Miller in \cite{Miller1984}. It says  that 
 $\map_{\ast}(B{\mathbb  Z}/p, Y)$ is a contractible space for all finite CW-complexes $Y$.  
 Another example, and one which serves as a motivation  of this paper, is the description
 of the structure of $n$-fold loop spaces accomplished by  Stasheff \cite{Stasheff1963}, May \cite{May1972}, 
 Bordman-Vogt \cite{J.M.Boardman1973}, Segal \cite{Segal1974}, Bousfield \cite{Bousfield1992}, and others. 
 Their work establishes criteria for recognizing if a space $X$ has a homotopy 
 type of a mapping space $\Omega^{n}Y = \map_{*}(S^{n}, Y)$ for some $Y$, as well as methods for reconstructing 
 the space $Y$ out of $X$.  Various approaches led to different ways of expressing these results, in the language of operads, PROPs, theories, special  simplicial spaces, Eilenberg-MacLane objects etc.  
 However, as much as they are different, there is a common feature they share. In each case a space $X$ is shown to 
 be an $n$-fold loop space if it admits certain maps between finite Cartesian products of $X$:
 $$X^{k}\to X^{l}$$ 
 (or, in the case of Segal and Bousfield's work maps $X[k]\to X[l]$, 
 where $X[k]\simeq X^{k}$, $X[l]\simeq X^{l}$),  and if compositions of these 
 maps satisfy appropriate relations. In effect this shows that spaces 
 of the homotopy type of loop spaces can be thought of as a kind of  algebraic objects in the category
 $\spaces_{\ast}$ of pointed topological spaces. 

Inspired by these results in this paper we investigate spaces $A$ with the property that 
mapping spaces $\map_{\ast}(A, Y)$ can be described as spaces equipped with some finitary algebraic structure. 
To make it precise, we assume that there exists an algebraic theory  $\TT$ in the sense of Lawvere, such that 
the homotopy category of algebras over $\TT$ is equivalent to the homotopy category of mapping spaces 
$\map_{\ast}(A, X)$. As we explain in Section \ref{ALG THEORIES SEC}, the language of algebraic theories subsumes all formalisms used to describe loop spaces mentioned above, so existence of an operad, PROP etc. characterizing 
mapping spaces implies existence of an algebraic theory satisfying the above condition. If such an algebraic theory 
exists then we say that $A$ is a  \emph{detectable} space (see \ref{detectable} for the precise definition). The results on 
the recognition of loop spaces imply that spheres $S^{n}$ are detectable for all $n\geq 0$, and their direct generalization 
gives that spaces of the form $\bigvee_{m} S^{n}$ are detectable for all $m, n\geq 0$. 
Our main result can be  stated as follows:

\begin{theorem}\label{MainThm}
Let $A$ be a detectable finite pointed CW-complex, then 
$$H_*(A, \integers)\isom H_*(\textstyle{\bigvee_m S^n}, \integers)$$
for some $m, n\geq 0$.
\end{theorem}

As a consequence we obtain 

\begin{corollary}
If $A$ is a simply connected finite pointed CW-complex then 
$A$ is detectable if and only if $A\simeq \bigvee_m S^n$ for some $m\geq 0$, $n>1$. 
\end{corollary}

We note that the class of detectable spaces is larger if we consider also infinite CW-complexes. For example
Sartwell \cite{SartwellThesis} showed that spheres localized at a set of primes are detectable.

\ 

{\bf Organization of the paper.}
In section \ref{ALG THEORIES SEC} we introduce algebraic theories, their algebras and describe how they can be used to detect mapping spaces. The main idea of the proof of Theorem \ref{MainThm} is to show that if a space $A$ is detectable then mapping spaces $\map_*(A,X)$ are preserved by localization functors. In sections \ref{LGOOD SEC} and \ref{COMMUTATION SEC} we show that this preservation property imposes certain conditions on the space $A$. Lastly, in section \ref{MAIN THM SEC} we complete the proof of Theorem \ref{MainThm}.

\

{\bf Acknowledgement.}
I would like to thank my advisor, Bernard Badzioch, for his guidance that led to this result, and his comments through many versions of this paper. An additional thanks to Emmanuel Dror Farjoun for the reference that inspired this work as well as some helpful conversations.

\section{Algebraic Theories}
\label{ALG THEORIES SEC}

\begin{definition}
\label{STRICT ALG}
An algebraic theory $ \TT$ is  a  simplicial category 
with objects $T_0, T_1, \dots$  such that $T_n$ is the $n$-fold categorical product of 
$T_1$. In particular $ T_0$ is the terminal object in $ \TT$. 
We assume that it is also  the initial object.
   
Given an algebraic theory $\TT$, a $\TT$-algebra 
is a product-preserving  simplicial functor $\Phi \colon \TT\to\sSet_*$ 
where $\sSet_*$ denotes the category of pointed simplicial sets. 
\end{definition}

We will denote by $ \alg^{\TT}$ the category of all strict 
$ \TT$-algebras with natural transformations of functors 
as morphisms. We have a forgetful functor 
$$U\colon  \alg^{\TT} \to \sSet_{*}$$
given by $U(\Phi) = \Phi(T_{1})$. 
We will say that the geometric realization  of the simplicial set $U(\Phi)$ is the 
\emph{underlying space} of the algebra $\Phi$. Intuitively, giving a $\TT$-algebra 
$\Phi$ amounts to defining a certain algebraic structure, determined by  $\TT$,  
on the underlying space of $\Phi$.

Algebraic theories appear naturally as a tool for describing homotopy structures on spaces. 
Out of the formalisms characterizing loop spaces mentioned in Section \ref{INTRO SEC} is
the work of Boardman and Vogt  explicitly uses algebraic theories. Furthermore,  any operad $\mathcal C$ 
defines an algebraic  theory $\TT_{\mathcal C}$ as follows. Take the monad $ F\colon \sSet\to\sSet$ associated 
with the operad $\mathcal C$ and consider free algebras $F(\bigvee_{k} S^{0})$ over $F$ for 
$k\geq 0$. These algebras define a full subcategory $\DD$ in the category of all algebras over $F$. 
The algebraic theory  $\TT_{\mathcal C}$ is isomorphic to the opposite category of $\DD$. 
One can check that the category of algebras over the operad $\mathcal C$ is isomorphic to the 
category of strict $\TT_{\mathcal C}$-algebras. A similar reasoning can be used to construct an algebraic 
theory $\TT$ associated with a given PROP in such a way, that spaces equipped with an action of the PROP 
coincide with spaces  equipped with a structure of a strict $\TT$-algebra.  
The relationship between algebraic theories and special simplicial spaces of Segal or Eilenberg-MacLane objects of Bousfield is somewhat more involved, but it has been established by Badzioch in \cite{Badzioch2005}. 

Algebraic theories have been used to detect mapping spaces via an equivalence between the categories of algebras and the category of topological spaces with a particular model category structure. The category of algebras $\alg^{\TT}$ over a theory $\TT$ can be equipped with a model category structure where a natural transformation $\Phi\to\Psi$ is a weak equivalence or fibration if the induced map $\Phi(T_1)\to\Psi(T_1)$ is a weak equivalence or fibration respectively in the category of simplicial sets. Given a $CW$-complex $A$, denote by ${\mathbf R}^{A}\top_{\ast}$ the category $\top_{\ast}$ of pointed topological spaces as a simplicial category (with $\hom_{\mathbf{R}^A\top_*}(X,Y)$ the simplicial mapping space $\sing(\map_*(X,Y))$ which we will denote by $\map^{\Delta}_*(X,Y)$), taken with the model category structure where fibrations are Serre fibrations and a weak equivalences are maps $f\colon X \to Y$ that induce weak equivalences on simplicial mapping complexes
$f_{\ast}\colon \map^{\Delta}_{\ast}(A, X) \to \map^{\Delta}_{\ast}(A, Y)$. In other words,  ${\mathbf R}^{A}\top_{\ast}$
is obtained by taking the right Bousfield localization of the usual model category structure on $\top_{\ast}$ with respect 
to the space $A$ as in \cite[5.1.1]{Hirschhorn2009}.

\begin{definition}\label{detectable} $A$ is \emph{detectable} if there exists an algebraic theory $\TT$ such that there is a Quillen equivalence
$$\tilde{B}:\alg^{\TT}\rightleftarrows R^{A}\top_*:\tilde{\Omega}$$
with underlying space $U(\tilde{\Omega}X)$ naturally weakly equivalent to $\map_*(A,X)$.
\end{definition}

The Quillen equivalence $(\tilde{B},\tilde{\Omega})$ means that any mapping space $\map_*(A,X)$ can be equipped (up to weak equivalence) with a $\TT$-algebra structure. Moreover, if $\Phi$ is a $\TT$-algebra, then $U(\Phi)\simeq\map_*(A,\tilde{B}\Phi)$, so any $\TT$-algebra comes from a mapping space.

\begin{example}
Given a pointed space $A$ there is a canonical algebraic theory $\TT^A$ such that 
$$\hom_{\TT^{A}}(T_n, T_m)=\map^{\Delta}_*\left(\bigvee^m A, \bigvee^n A\right).$$

The functor $\Omega^A:R^A\top_{\ast}\to\alg^{T^A}$ defined by $\Omega^{A}(X)(T_{k})=\map_{\ast}^{\Delta}\left(\bigvee_kA, X\right)$ has a left adjoint $B^A$. The adjunction is not always a Quillen equivalence, but in particular,
$$B^{S^n}:\alg^{\TT}\rightleftarrows R^{A}\top_*:\Omega^{S^n}$$
is a Quillen equivalence \cite{BadziochChungVoronov2007} as well as 
$$B^{S^n_P}:\alg^{\TT}\rightleftarrows R^{A}\top_*:\Omega^{S^n_P}$$
where $S^n_P$ is the sphere localized at a set of primes \cite{SartwellThesis}. It follows that spheres and spheres localized at a set of primes are detectable, as well as their wedges of the same dimension.
\end{example}
 
For our purposes it will be useful to rephrase the notion of detectability using homotopy algerbas in place of strict algebras. A \emph{homotopy algebra} over an algebraic theory $\TT$ is a simplicial functor $\TT\to\sSet_*$ that preserves products up to a weak equivalence. We denote the category of homotopy algebras over $\TT$ by $\halg^{\TT}$.  By the rigidification theorem of algebras \cite[1.3]{Badzioch2002}, there is a Quillen equivalence $\halg^{\TT}\rightleftarrows\alg^{\TT}$ between the categories of homotopy and strict algebras over $\TT$. Using this fact, we obtain:
 
  \begin{proposition}\label{detectablehomotopyalgs}
  A space $A$ is \emph{detectable} if and only if there is some algebraic theory $\TT$ such that there is a Quillen equivalence
 $$\tilde{B}:\halg^{\TT}\rightleftarrows R^{A}\top_*:\tilde{\Omega}$$
with underlying space $U(\tilde{\Omega}X)$ weakly equivalent to $\map_*(A,X)$
  \end{proposition}

\section{L-Good Spaces}
\label{LGOOD SEC}

As the first step towards proving Theorem \ref{MainThm} we will show that if $A$ is detectable, then mapping spaces out of $A$ behave like loop spaces with respect to localization functors. Let $f:Z\to W$ be a map of pointed spaces, a pointed space $X$ is said to be \emph{f-local} if the induced map
$$f_*:\map_*(W,X)\to\map_*(Z,X)$$
is a homotopy equivalence. The $f$-localization of $X$, denoted $L_fX$, is a functorial construction of the $f$-local space closest to $X$. Explicitly, $L_fX$ is equipped with a coaugmentation map $X\to L_fX$ and given any $f$-local space $Y$ with a map $g:X\to Y$, there is a factorization unique up to homotopy:

$$\begin{tikzpicture}
\matrix (m) 
[matrix of math nodes, row sep=2em, column sep=3em, text height=1.5ex, text depth=0.25ex]
{
X &  L_{f}X\\
 &   Y\\
};
\path[->]
(m-1-1) edge  (m-1-2)
(m-1-1)   edge (m-2-2)
;
\path[->, dashed]
(m-1-2)  edge  (m-2-2)
;
\end{tikzpicture}$$
In the particular case where $f:Z\to\ast$ is a constant map, we denote $L_fX$ by $P_ZX$ and call it the \emph{Z-nullification} of $X$. If $X$ is local with respect to the map $Z\to\ast$ that is, $\map_*(Z,X)\simeq\ast$ we say $X$ is a \emph{$Z$-null} space.

It is well known \cite[3.A.1]{Farjoun1995} that loop spaces commute with localization functors in the sense that for any map $f$ of pointed spaces and $X$ a pointed space there is a weak equivalence $$L_f\Omega X\simeq\Omega L_{\Sigma f}X.$$ 
This property is generalized with the following definition.
\begin{definition} A space $A$ is \emph{L-good} if for any map $f$ the localization functor $L_f$ preserves mapping spaces out of $A$, that is,
$$L_f\map_*(A,X)\simeq\map_*(A,Y).$$ 
\end{definition}

\begin{proposition}
If the space $A$ is detectable, then $A$ is an $L$-good space.
\end{proposition}

\begin{proof}
Assume $A$ is detectable by the algebraic theory $\TT$ and let $X$ be any space. Then $\map_*(A,X)$ is weakly equivalent to the underlying space of some algebra $\Phi$ over $\TT$. $L_f$ preserves products up to a weak equivalence, so $L_f\circ\Phi$ is a homotopy algebra with underlying space $L_f\map_*(A,X)$. By detectability of $A$, the underlying space of any homotopy algebra over $\TT$, in particular $L_f\circ\Phi$, is up to weak equivalence a mapping space out of $A$. We then obtain $$L_f\map_*(A,X)\simeq \map_*(A,Y)$$ for some space $Y$.
\end{proof}
 If $A$ is an $L$-good finite CW-complex, then the rational homotopy type of $A$ is determined.

\begin{theorem}\label{rationaltype}\cite[1.2]{BadziochDorabiala2010}
\label{rationaltype}
Let $A$ be a finite, connected, pointed CW-complex such that for some $p>q>0$ we have $H^p(A,\mathbb{Q})\neq 0\neq H^q(A,\mathbb{Q})$. Then $A$ is not an L-good space.
\end{theorem}

 This result can be extended to rule out some cases of torsion.

\begin{proposition}\label{torsion}
Let $A$ be a finite, pointed, connected $CW$-complex. If $A$ is L-good and $H_n(A,\integers)$ contains a torsion-free element for some $n>0$ then $H_q(A,\integers)=0$ for all $0<q<n$.
\end{proposition}

In order to prove this fact, we need to introduce the concept of celluarization, a generalization of CW-approximation. We say that a space $X$ is \emph{$A$-cellular} if $X$ is in the smallest class containing $A$ which is closed under homotopy colimits and weak equivalences. The space $CW_AX$ associated to a space $X$ is the $A$-cellular space closest to $X$. More precisely, we say that a map $f:X\to Y$ is an \emph{$A$-equivalence} if the induced map $$\map_*(A, X)\to\map_*(A,Y)$$ is a homotopy equivalence.
Given a space $X$, the $A$-cellular approximation of $X$ is an $A$ cellular space $CW_AX$ equipped with an $A$-equivalence $X\to CW_AX$. The assignment $X\mapsto CW_AX$ defines a functor $CW_A:\spaces_*\to\spaces_*$ \cite[2.B]{Farjoun1995}.

 Given any map $Y\to X$ where $Y$ is $A$-cellular there is a factorization
$$\begin{tikzpicture}
\matrix (m) 
[matrix of math nodes, row sep=2em, column sep=3em, text height=1.5ex, text depth=0.25ex]
{
Y &  CW_AX\\
 &   X\\
};
\path[->]
(m-1-2)  edge  (m-2-2)
(m-1-1)   edge (m-2-2)
;
\path[->, dashed]
(m-1-1) edge  (m-1-2)

;
\end{tikzpicture}$$

unique up to homotopy. The following theorem describes the relationship between cellularization and nullification.

\begin{theorem}\label{fibthm}\cite[3.B.2]{Farjoun1995}
If $[A,X]\simeq\ast$, then the sequence
$$CW_AX\to X\to P_{\Sigma A}X$$
is a fibration sequence.
\end{theorem}

Recall that Generalized Eilenberg-Maclane spaces (GEMs) are spaces weakly equivalent to a product of Eilenberg-Maclane spaces $\tilde{\Pi}_n K(G_n,n)$. Localization and cellularization preserves GEMs in particular cases.

\begin{lemma}\label{GEMlem}\cite[5.B.1.1]{Farjoun1995}
If $P_{\Sigma A}X\simeq\ast$ then $P_{\Sigma^2 A}X$ is a GEM.
\end{lemma}

\begin{lemma}\label{GEMlem2}
Let $A$, $X$ be connected spaces such that $\map_*(A,X)\simeq K(G,1)$ for some group $G$, then $CW_A(X)$ is a GEM.
\end{lemma}

\begin{proof}
Since $[A, CW_AX]\isom[A,X]=\pi_0(\map_*(A,X))=\ast$, by Theorem \ref{fibthm} we have a fibration sequence
$$CW_ACW_AX\to CW_AX\to P_{\Sigma A} CW_AX.$$
The map $CW_ACW_AX\to CW_AX$ is a weak equivalence, which gives $P_{\Sigma A} CW_AX\simeq\ast$ and by Lemma \ref{GEMlem} $P_{\Sigma^2 A} CW_AX$ is a GEM. It then suffices to show $CW_AX\simeq P_{\Sigma^2 A} CW_AX$, or equivalently that $CW_AX$ is a $\Sigma^2A$-null space. This holds since
$$\map_*(\Sigma^2A,CW_A X)\simeq\Omega^2\map_*(A,X)\simeq\Omega^2 K(G,1)\simeq\ast.$$

\end{proof}

\begin{proof} [Proof of Proposition \ref{torsion}]

Assume that $A$ is a finite, pointed, connected $CW$-complex such that $H_n(A,\integers)$ contains a torsion-free element and that there exists $0<m<n$ such that $H_m(A,\integers)\neq 0$. For the sake of contradiction, suppose $A$ is $L$-good, then by Theorem \ref{rationaltype}, $H_m(A;\integers)$ must be torsion.

Consider the map $f:S^2\vee(\bigvee_{p\in P}S^1)\to (\bigvee_{p\in P}S^1)$, where $P$ is the set of all primes, such that $f|_{S^2}=\ast$ and $f|_{S^1_p}:S^1_p\to S^1_p$ is the degree $p$ map $S^1\to S^1$. The localization $L_fX$ of a space $X$ has trivial homotopy groups $\pi_i(L_fX)$ for $i>1$ and has no torsion in the fundamental group and so $$L_fX=K(\pi_1(X)/\mathrm{Tor}(\pi_1(X)),1).$$ We will show that $A$ is not $L$-good with respect to the map $f$. Indeed, otherwise there would exist a space $Y$ (which can be assumed to be $A$-cellular) such that 
$$L_f\map_*(A, K(\integers, n+1))\simeq \map_*(A,Y).$$
However, $L_f\map_*(A, K(\integers, n+1))\simeq K(H^n(A;\integers)/\mathrm{Tor}(H^n(A;\integers)),1)$. By assumption, $H_n(A,\integers)$ has a nontrivial torsion-free part so $H_n(A,\integers)\isom \integers^k\oplus \mathrm{Tor}(H_n(A,\integers))$ for some $k>0$. By the universal coefficient theorem, the torsion-free part of $H^n(A;\integers)$ is isomorphic to the torsion-free part of $H_n(A,\integers)$. This gives
$$\map_*(A,Y)\simeq K(\integers^k,1).$$
By Lemma \ref{GEMlem2} we obtain, $Y\simeq\tilde{\Pi}_{i=1}^{\infty}K(G_i,i)$ and so
$$\map_*(A,\tilde{\Pi}_{i}K(G_i,i))\simeq K(\integers^k,1).$$
Restricting our attention to the fundamental groups, we get $\oplus_i\tilde{H}^{i-1}(A,G_i)\isom\integers^k$ where each of the summands on the left side is torsion, hence trivial, with the exception of $i=n+1$. We obtain
$$\integers^k\isom H^n(A; G_{n+1})\isom\hom(H_n(A,\integers), G_{n+1})\oplus \ext(H_{n-1}(A,\integers),G_{n+1})$$
which implies
$$\integers^k\isom \hom(H_n(A,\integers), G_{n+1})$$
as $\ext(H_{n-1}(A,\integers),G_{n+1})$ is torsion.  Note that $\hom(H_n(A,\integers), G_{n+1})=\hom(\integers^k\oplus Tor(H_n(A,\integers)),G_{n+1})$ and so $\integers^k\isom \bigoplus_{i=1}^k G_{n+1}\oplus(\mathrm{torsion})$ which implies $G_{n+1}\isom\integers$.

Now, $H_m(A,\integers)$ is torsion, so by the universal coefficient theorem it is a direct summand of $H^{m+1}(A,\integers)$ so we have
$$H_m(A,\integers)\subseteq H^{m+1}(A,\integers)\isom\pi_{n-m}\map_*(A,K(\integers,n+1)).$$
Since $G_{n+1}\isom\integers$, the space $\map_*(A, K(\integers, n+1))$ is a retract of $\map_*(A,Y)$ and so 
$$H_m(A,\integers)\subseteq \pi_{n-m}(\map_*(A,Y))=\pi_{n-m}(K(\integers^k,1))$$
which is a contradiction. This shows that the space $Y$ cannot exist and so $A$ is not L-good.
\end{proof}

\begin{proposition}
Let $A$ be a finite, pointed, connected CW-complex. If $H_n(A,\integers)\isom\integers^k\oplus T$ for $k>0$ and $T$ is a non-trivial torsion group, then $A$ is not $L$-good.
\end{proposition}

\begin{proof}
We can use the same argument as the previous proposition to show that if $f$ is the map as in the proof of Proposition \ref{torsion} and $L_f\map_*(A, K(\integers, n+1))\simeq\map_*(A, Y)$ then
$$T\subseteq\pi_0\map_*(A,K(\integers,n+1))\subseteq\pi_0(\map_*(A,Y))=\pi_0(K(\integers^k,1))$$
which is a contradiction.
\end{proof}

\begin{corollary}\label{homcor}
If $A$ is a finite, pointed, connected, CW-complex that is $L$-good and $H_n(A,\integers)$ contains a torsion-free element for some $n>0$, then
$$H_q(A;\integers)\isom\left\{\begin{array}{c c}
0 & \text{ if }0<q<n\\
\integers^k & \text{ if }q=n\\
\mathrm{torsion} & \text{ if } q>n\\
\end{array}\right.$$
\end{corollary}

\section{Commutation of Localization}
\label{COMMUTATION SEC}

By definition, if $A$ is an $L$-good space then the class of mapping spaces $\map_*(A,X)$ is preserved by localization functors. If $A=S^n$, the homotopy type of the space $L_f\map_*(A,X)$ can be described more precisely as follows:
\begin{theorem}\label{LoopComm}\cite[3.1]{Bousfield1994} Let $f:Z\to W$ be any map in $\spaces_*$ and $X$ a pointed space. The natural map
$$L_f\map_*(S^n,X)\to\map_*(S^n, L_{S^n\wedge f}X)$$
is a homotopy equivalence.
\end{theorem}

This result motivates the following definition:

\begin{definition}
A space $A$ is \emph{strongly L-good} if for any space $X$ and any map $f$, there exists a map $g$ such that $$L_f(\map_*(A,X))\simeq \map_*(A,L_g(X)).$$
\end{definition}

We will show that the following holds:

\begin{theorem}\label{FirstThm}
If $A$ is a finite CW complex that is strongly $L$-good, then 
$$H_*(A)\isom H_*\left(\bigvee_lS^n\right)$$ for some $n\ge 1$, $l\ge 0$.
\end{theorem}

We will be utilize the following property of localizations of Eilenberg-Maclane Spaces.

\begin{proposition}\label{EMcor}\cite[3.1]{Casacuberta1998}
Let $f$ be any map and $n\ge 1$, then
$$L_f(K(\integers,n))\simeq K(G,n)$$
for some abelian group $G$.
\end{proposition}

\begin{proof} [Proof of Theorem \ref{FirstThm}]

We will argue by contradiction. Assume that $A$ is strongly $L$-good, then by Corollary \ref{homcor}, we have 
$$H_q(A)\isom\left\{\begin{array}{c c}
0 & \text{ if }0<q<n\\
\integers^l & \text{ if }q=n\\
\mathrm{torsion} & \text{ if } q>n\\
\end{array}\right..$$

Assume that $H_q(A)\neq 0$ for some $q>n$ and let $m>n$ be the largest integer such that $H_m(A)\neq 0$.

By the universal coefficients theorem, we get
$$\pi_i\map_*(A,K(\integers,n+m+1))\isom \hom(H_{n+m+1-i}(A);\integers)\oplus \ext(H_{n+m-i}(A),\integers)$$ 
and our assumptions on the homology of $A$ give
$$\pi_i(\map_*(A,K(\integers,n+m+1))=\left\{\begin{array}{c c}
0 & i<n\\
H_{m}(A) & i=n\\
\text{torsion} & n<i<m+1\\
\integers^l & i=m+1\\
0 &i>m
\end{array}\right.$$.
Let $P_{S^{n+1}}$ be the $S^{n+1}$ nullification functor, we have
$$\pi_i(P_{S^{n+1}}\map_*(A,K(\integers,n+m+1))=\left\{\begin{array}{c c}
0 & \text{otherwise}\\
H_{m}(A) & i=n\\
0 &i>m
\end{array}\right.$$
and hence $P_{S^{n+1}}\map_*(A,K(\integers,n+m+1))\simeq K(H_m(A), n)$. By the assumptions on $A$ we have
$$P_{S^{n+1}}\map_*(A,K(\integers,n+m+1))\simeq \map_*(A,  L_g(K(\integers, n+m+1))).$$
Applying Proposition \ref{EMcor}, the localization of an Eilenberg-Maclane space is again an Eilenberg-Maclane space, so we obtain
$$\map_*(A,  L_g(K(\integers, n+m+1)))\simeq\map_*(A, K(G,n+m+1))$$
for some group $G$. This gives
$$K(H_m(A),n)\simeq\map_*(A,K(G,n+m+1)).$$
We obtain
$$0\isom\pi_{m+1}(K(H_m(A),n))\isom H^n(A;G)\isom \hom(H_n(A),G)\oplus\ext(H_{n-1}(A),G)$$
Since $H_n(A)$ is a free abelian group this implies that $G$ is trivial, so $$K(H_m(A),n)\simeq\map_*(A, K(G,n+m+1))\simeq\ast$$ which is a contradiction.
\end{proof}

\begin{corollary}
If $A$ is a simply connected finite CW-complex that is strongly $L$-good space, then $A$ has the homotopy type of a wedge of spheres $\bigvee_l S^k$ for some $k>0$, $l>0$.\end{corollary}
\begin{proof}
 By the Hurewicz isomorphism $h:\pi_nX\to H_nX\isom\integers^l$, we have maps $\alpha_i: S^n\to X$ for $i=1,\dots,l$ such that $\{h(\alpha_i)|i=1,\dots,l\}$ are generators of $H_n(X)$. The map $\bigvee_i\alpha_i:\bigvee_i S^n\to X$ induces isomorphisms of homology groups, and by the Whitehead theorem, it is a homotopy equivalence.
\end{proof}

\section{Proof of the Main Theorem}
\label{MAIN THM SEC}
In view of Theorem \ref{FirstThm} in order to complete the proof of Theorem \ref{MainThm} it will suffice to show that the following holds:

\begin{theorem}\label{detectableLgood}
If $A$ is a detectable, finite $CW$-complex then $A$ is strongly $L$-good.
\end{theorem}

Our proof will parallel Bousfield's argument showing that loop spaces commute with localization functors (Theorem \ref{LoopComm}). The argument relies on Segal's machinery to detect loop spaces which we will substitute with the Quillen equivalence $(\tilde{B}, \tilde{\Omega})$ given in the definition of a detectable space (Definition \ref{detectable}, Proposition \label{detectablehomotopyalgs}). The Quillen equivalence descends to a derived equivalence on homotopy categories. We will choose to work within the homotopy category, and will use the following lemma.

\begin{lemma}\label{idempotentfunctorlem}
If $F,G: \ho\CC \to \ho\DD$ are coaugmented, idempotent functors such that $c \isom F(C)$ if and only if $c\isom G(c)$ then $F$ and $G$ are naturally isomorphic.
\end{lemma}
\begin{proof}
For any $c\in\CC$ we have the following diagram.
$$\begin{tikzpicture}
\matrix (m) 
[matrix of math nodes, row sep=5em, column sep=10em, text height=1.5ex, text depth=0.25ex]
{
c & F(c) & F(F(c))\\
 & G(c) & F(G(c))\\
};
\path[->]
(m-1-1) edge node[above]{\footnotesize$\alpha$}  (m-1-2)
(m-1-1)   edge node[below]{\footnotesize$\beta$} (m-2-2)
(m-1-2) edge node[above]{$\simeq$}node[below]{$\omega$} (m-1-3)
(m-1-2) edge node[above]{\footnotesize$F(\beta)$} (m-2-3)
(m-2-2) edge node[below]{\footnotesize$\gamma$} node[above]{$\simeq$} (m-2-3)
;
\path[->, dashed]
(m-1-2) edge[bend left] node[right]{\small$\delta$} (m-2-2)
(m-1-3) edge[bend left] node[right]{\small$F(\delta)$}  (m-2-3)
;
\path[->, densely dotted]
(m-1-2) edge[bend right] node[left]{\small$\lambda$}  (m-2-2)
(m-1-3) edge[bend right] node[left]{\small$F(\lambda)$} (m-2-3)
;
\end{tikzpicture}$$
where $\alpha$, $\beta$, $\gamma$, and $\omega$ are defined by the coaugmentation maps. Since $F(c)\simeq F(F(c))$, $G(c)\simeq F(G(c))$ and $\gamma$ is an isomorphism. We define $\delta=\gamma^{-1}\circ F(\beta)$. We first claim that $\delta$ is unique up to weak equivalence, to show this assume $\lambda$ is another map $F(c)\to G(c)$. Then we get $F(\lambda)\circ\omega\simeq F(\beta) \simeq F(\delta)\circ\omega$ and since $\omega$ is a weak equivalence we have $F(\lambda)\simeq F(\delta)$ and hence $\lambda\simeq\delta$. The factorization of $\beta$ through $\alpha$ is unique up to homotopy. We can similarly show the factorization of $\alpha$ through $\beta$ is unique up to homotopy and thus proving $\delta$ is an isomorphism.

\end{proof}
 
 We will use the following fact which is a direct consequence of the mapping space smash product adjunction. 

\begin{lemma}\label{flocallem}
$X$ is $(A\wedge f)$-local if and only if $\map_*(A,X)$ is $f$-local.
\end{lemma}

\begin{proof}[Proof of Theorem \ref{detectableLgood}]
Since $A$ is detectable, there is a derived equivalence on the level of homotopy categories.
$$\tilde{B}_{\text{der}}:\ho\alg^T\rightleftarrows \ho R^A\top_*:\tilde{\Omega}_{\text{der}}.$$

We first claim the functors $\tilde{B}_{\text{der}} L_f \tilde{\Omega}_{\text{der}}$ and $L_{A\wedge f}$ are both coaugmented, idempotent functors with the same images, thus by lemma \ref{idempotentfunctorlem} are the naturally isomorphic. We have natural maps
$$X\overset{\simeq}\longleftarrow \tilde{B}_{\text{der}}\tilde{\Omega}_{\text{der}}X\longrightarrow \tilde{B}_{\text{der}} L_f \tilde{\Omega}_{\text{der}}X$$
where the latter is induced from the coaugmentation map of $L_f$, hence $\tilde{B}_{\text{der}} L_f \tilde{\Omega}_{\text{der}}$ is coaugmented. 

We now show for $X$ in $R^A\top_*$, $\tilde{B}_{\text{der}}L_f \tilde{\Omega}_{\text{der}}X\simeq X$ if and only if $L_{A\wedge f}X\simeq X$. As $X\simeq \tilde{B}_{\text{der}}\tilde{\Omega}_{\text{der}}X$, the spaces $\tilde{B}_{\text{der}}L_f \tilde{\Omega}_{\text{der}}X$ and $X$ are weakly equivalent if and only if there is a weak equivalence of algebras $L_f \tilde{\Omega}_{\text{der}} X\simeq \tilde{\Omega}_{\text{der}}X$. By the model category structure on the category of algebras over a theory, $L_f \tilde{\Omega}_{\text{der}} X\simeq \tilde{\Omega}_{\text{der}}X$ if and only if the underlying spaces $L_f\map_*(A,X)$ and  $\map_*(A,X)$ are weakly equivalent which implies $\map_*(A,X)$ must be $f$-local. By lemma $\ref{flocallem}$, we obtain that $\tilde{B}_{\text{der}}L_f \tilde{\Omega}_{\text{der}}X\simeq X$ if and only if $L_{A\wedge f}X\simeq X$.

The functors $\tilde{B}_{\text{der}} L_f \tilde{\Omega}_{\text{der}}$ and $L_{A\wedge f}$ are then naturally isomorphic, so by composing with $\tilde{\Omega}_{\text{der}}$, the following algebras are weakly equivalent
$$L_f\tilde{\Omega}_{\text{der}}X\simeq \tilde{\Omega}_{\text{der}}\tilde{B}_{\text{der}} L_f \tilde{\Omega}_{\text{der}}X\simeq \tilde{\Omega}_{\text{der}}L_{A\wedge f}X$$
with associated underlying spaces $$L_f\map_*(A,X)\simeq \map_*(A, L_{A\wedge f}X)$$
completing the proof that $A$ is strongly $L$-good.
\end{proof}

\bibliographystyle{plain}
\bibliography{ReferencesPaper}

\end{document}